\documentclass{amsart}
\usepackage{amssymb}

\newtheorem{thm}{Theorem}[section]
\newtheorem{cor}[thm]{Corollary}
\newtheorem{lem}[thm]{Lemma}

\newtheorem{prop}[thm]{Proposition}
\theoremstyle{definition}
\newtheorem{defn}[thm]{Definition}
\theoremstyle{remark}
\newtheorem{rem}[thm]{Remark}
\numberwithin{equation}{section}

\begin{document}

\title[On the curvature ODE]{On the curvature ODE associated to the Ricci flow}
\author{Atreyee Bhattacharya}

\address{department of mathematics,
Indian Institute of Science, Bangalore 560012, India}
\email{atreyee@math.iisc.ernet.in}


\thanks{This work was partially supported by UGC}


\vspace{3mm}

\begin{abstract}
In the vector space of algebraic curvature operators we study the reaction ODE $$\frac{dR}{dt} = R^2+R^{\#}= Q(R)$$  which is associated to the evolution equation of the Riemann curvature operator along the Ricci flow. More precisely, we analyze the stability of a special class of zeros of this ODE up to suitable normalization. In particular, we show that the ODE is unstable near the curvature operators of the Riemannian product spaces $M \times \mathbb{R}^k, \ k \geq 0$ where $M$ is an Einstein (locally) symmetric space of compact type and not a spherical space form when $k=0.$
\end{abstract}


\maketitle

\section{Introduction}
The Ricci flow, introduced by Richard Hamilton ~\cite{H1}, is a geometric evolution equation on Riemannian manifolds. Let $(M,g)$ be a compact Riemannian manifold without boundary. A smooth $1$-parameter family $g(t), t \in [0,T]$ of Riemannian metrics on $(M,g)$ is said to satisfy the Ricci flow starting at $g$ if the following equation holds
{\small \begin{equation}\label{flow}
\frac{\partial g(t)}{\partial t} = -{2Ric_{g(t)}}, \ \ g(0) = g.
\end{equation}}
One also considers normalized Ricci flows where the flow $\tilde{g}(t)$ is of the form $\tilde{g}(t)= f(x,t)g(t)$
where $f : M \times [0,T] \rightarrow \mathbb{R}^{+}$ is a positive smooth function. A well studied example is when $f(x,t)= vol(g(t))^{\frac{-2}{n}}.$ We will refer to the corresponding equation
{\small \begin{equation}\label{flowtwo}
\frac{\partial \tilde{g}(t)}{\partial t} = F(\tilde{g}(t))
\end{equation}}
\noindent
as well as the solution curve $\tilde{g}(t)$ as a normalized Ricci flow abbreviated as NRF.
A Riemannian metric $g_0$ is said to be a \textit{stable} fixed point of the NRF ~\eqref{flowtwo} if the normalized flow $\tilde{g}(t)$ starting at $g_0$ is identically equal to $g_0$ and there exists a neighborhood of $g_0$ such that the normalized flow starting at any point in that neighborhood converges to $g_0$. We will call such a neighborhood \textit{stable}. One observes that suitable normalizations of the Ricci flow are needed as fixed points of the unnormalized flow \eqref{flow} are Ricci-flat i.e., do not even include the round metric, whereas Einstein metrics that are not Ricci-flat can be realized as fixed points of some NRF.

One would like to find the largest possible stable neighborhoods of stable fixed points of the NRF ~\eqref{flowtwo}.  For instance, results due to B\"ohm-Wilking ~\cite{BW} and Brendle-Schoen ~\cite{BS} show that the round metric on sphere is a stable fixed point of a suitable NRF and neighborhoods of this metric consisting of Riemannian metrics with $2$-positive curvature operators and strictly quarter pinched curvature respectively are stable.

To understand the stability of a fixed point $g_0$ of ~\eqref{flowtwo}, we go back to ~\eqref{flow} and study the reaction ODE associated to the evolution of Riemann curvature operator along ~\eqref{flow}. In general, it may not be possible to infer the stability of $g_0$ solely by analyzing the behavior of the reaction ODE near the corresponding curvature operator $\mathfrak{R}_{g_0}$. But in some cases using Hamilton's maximum principle, one might be able to do so. This is briefly explained below.

Under the (unnormalized) Ricci flow \eqref{flow} the Riemann curvature operator {\small $\mathfrak{R} = \mathfrak{R}(t): \wedge^2 TM \rightarrow  \wedge^2 TM$} evolves by
{\small \begin{equation}\label{pde}
\frac{\partial \mathfrak{R}}{\partial t} = \Delta \mathfrak{R} + 2Q(\mathfrak{R})
\end{equation}}
\noindent
where {\small $$Q(\mathfrak{R})= \mathfrak{R}^2 + \mathfrak{R}^{\#}$$}
\noindent
is a homogenous quadratic expression in the components of {\small$\mathfrak{R}$}.
The ordinary differential equation corresponding to the reaction term of \eqref{pde} is the following homogeneous ODE of degree $2$
{\small \begin{equation}\label{reaone}
\frac{dR}{dt} = Q(R)
\end{equation}}

We denote the space of algebraic curvature operators i.e., the vector space of self-adjoint endomorphisms of {\small $\wedge^2 \mathbb{R}^n$} satisfying the Bianchi identity, by {\small $S^2_{B}(\wedge^2\mathbb{R}^n).$}
Hamilton's maximum principle ~\cite{H1} for tensors says that a closed convex $O(n)$-invariant subset $C$ of {\small $S^2_{B}(\wedge^2\mathbb{R}^n)$} which is invariant under \eqref{reaone}, is a Ricci flow invariant curvature condition i.e., metrics on compact manifolds whose curvature operators at each point lie inside $C,$ evolve under the Ricci flow into metrics with the same curvature property.

In order to use the maximum principle to show the stability of a fixed point $g_0$ of ~\eqref{flowtwo} one constructs a `pinching family' of closed convex $O(n)$ invariant sets around $\mathfrak{R}_{g_0}$ in {\small $S^2_{B}(\wedge^2\mathbb{R}^n).$} This has been accomplished for positive curvature operators by B\"ohm-Wilking ~\cite{BW} and for quarter pinched curvature operators by Brendle-Schoen ~\cite{BS}. In these cases the existence of a pinching family implies that $\mathfrak{R}_{g_0}$ is a stable fixed point or zero of a suitable normalization of the reaction ODE ~\eqref{reaone}.

In this paper we study the zeros of ~\eqref{reaone} for a specific normalization where the solution curves have constant scalar curvature one.
Let $s(R)$ denote the scalar curvature of {\small $R \in S^2_{B}(\wedge^2\mathbb{R}^n).$}  Let {\small $R(t) \in \{R \in S^2_{B}(\wedge^2\mathbb{R}^n): s(R) > 0\}$} be a solution curve to the ODE \eqref{reaone}. Then $\bar{R}(t)= \frac{R(t)}{s(R(t))}$ satisfies the following equation

{\small \begin{equation}\label{scalone}
\frac{d\bar{R}}{dt} = s(R)\bigg(Q(\bar{R})- \|Ric(\bar{R})\|^2\bar{R} \bigg)
\end{equation}}
\noindent
where we use
{\small $$\frac{d}{dt}(s(R)) = s(Q(R)) = \|Ric(R)\|^2$$}
\noindent
(which follows from formula \eqref{scal}) and the fact that
the open half space {\small $\{R \in S^2_{B}(\wedge^2\mathbb{R}^n): s(R) > 0\}$} is preserved by \eqref{reaone}.
Reparametrizing \eqref{scalone} suitably (see Remark \ref{repara}) one obtains the associated ODE
\begin{equation}\label{reatwo}
\begin{cases}
\frac{d\widetilde{R}}{d\tau} = Q(\widetilde{R})- \|Ric(\widetilde{R})\|^2\widetilde{R} \\
s(\widetilde{R})=1
\end{cases}
\end{equation}

One further observes that the modified ODE \eqref{reatwo} has the special property that curvature operator of any Einstein (locally) symmetric space with unit scalar curvature is a zero of it. More generally, curvature operator of the Riemannian product {\small $M \times \mathbb{R}^k,$} where $\mathbb{R}^k$ denotes the Euclidean $k$-space with $k \geq 0$ and $M$ is any (locally) symmetric Einstein manifold with unit scalar curvature, is a zero of \eqref{reatwo} ( See Proposition \ref{prop} for details).

Define the vector field $\widetilde{Q}$ on the hyperplane {\small $${{{\mathbf{S}}}}_{1} = \{R \in S^2_B(\wedge^2\mathbb{R}^n) \;:\; s(R) = 1\} \subset S^2_B(\wedge^2\mathbb{R}^n),$$}
\noindent
by {\small $$\widetilde{Q} : R \mapsto \widetilde{Q}(R) = Q(R)- \|Ric(R)\|^2R, \ \ R \in \mathbf{S}_1.$$}
\noindent
One first observes that
\begin{prop}
If $(M,g)$ is a closed Einstein manifold with unit (respectively zero) scalar curvature, then at each point $p \in M$, the curvature operator $\mathfrak{R}_p$ of $g$ is a zero of $\widetilde{Q}$ (respectively $Q$) if and only if $(M^n,g)$ is locally symmetric (respectively flat).
\end{prop}
In dimension $4$ we have the following stronger converse which is valid in a purely algebraic set up.
\begin{thm}\label{th1}
(1). The only algebraic curvature operator $R \in S^2_B(\wedge^2\mathbb{R}^4)$  which is a zero of $Q$ is the zero curvature operator. \\

(2). If $R \in S^2_B(\wedge^2\mathbb{R}^4)$ is an Einstein algebraic curvature operator (i.e., the Ricci tensor of $R$ is a scalar multiple of the identity of $\mathbb{R}^4$) which is a zero of $\widetilde{Q},$ then $R$ is the curvature operator of a (locally) symmetric space with unit scalar curvature.
\end{thm}
Our next result classifies the algebraic curvature operators in all dimensions, that are zeros of $\widetilde{Q}$ of Ricci type (i.e., with vanishing Weyl component). One notes that the curvature operator of any (locally) conformally flat manifold is of Ricci type.
\begin{thm}\label{th2}
For dimension $n+1,$ $n \geq3;$ if $R \in S^2_B(\wedge^2\mathbb{R}^{n+1})$ is an algebraic curvature operator of Ricci type which is not a multiple of the identity operator and $\widetilde{Q}(R) =0,$ then up to the action of the orthogonal group, $R$ is the curvature operator of {\small $S^n\times \mathbb{R}$} corresponding to the standard product metric normalized to unit scalar curvature.
\end{thm}
For dimensions greater than $4,$ the full set of zeros of $Q$ or all the Einstein zeros of $\widetilde{Q}$ are not known to us. Even in dimension $4$ it is not clear if there exists a zero of $\widetilde{Q},$ other than the curvature operator of {\small $S^2 \times \mathbb{R}^2,$} which is neither Einstein and nor of Ricci type.

Our main result is the following
\begin{thm}\label{main}
The reaction ODE \eqref{reaone} behaves unstably near the curvature operator of the Riemannian product {\small $M \times \mathbb{R}^k$} for $k \geq 0,$ where $M$ is an Einstein symmetric space with positive scalar curvature and not a spherical space form when $k=0.$
\end{thm}

The author would like to thank Harish Seshadri for suggesting this problem and his guidance. Thanks are also due to Thomas Richard for making some crucial observations related to this paper.

\section{Preliminaries}

Let {\small$\wedge^2\mathbb{R}^n$} denote the exterior $2$-product of {\small$\mathbb{R}^n.$}  The canonical inner product on {\small$\wedge^2\mathbb{R}^n$} induced by the standard inner product on {\small$\mathbb{R}^n$} is such that $\{e_i \wedge e_j\}_{i < j}$ is an orthonormal basis of {\small$\wedge^2\mathbb{R}^n$} whenever {\small$\{e_i\}^n_{i = 1}$} is an orthonormal basis of {\small$\mathbb{R}^n$}.
{\small$\wedge^2\mathbb{R}^n$} is identified with the Lie algebra $\mathfrak{so}(n)$ via the linear map {\small$x \wedge y : \mathbb{R}^n \rightarrow \mathbb{R}^n$} defined by
{\small$(x \wedge y) : z \mapsto \langle y,z\rangle x - \langle x,z\rangle y$}, for {\small$x,y,z \in \mathbb{R}^n$}. With this identification $\mathfrak{so}(n)$ has the inner product {\small $\langle{A},{B}\rangle = -\frac{1}{2}tr (AB).$}

The \textit{wedge product} {\small$A \wedge B$} of two linear endomorphisms {\small$A,B$} of {\small $\mathbb{R}^n$} is a linear endomorphism of {\small$\wedge^2\mathbb{R}^n$} defined by  {\small $$(A \wedge B)(v \wedge w ) = \frac{1}{2}({A(v) \wedge {B(w)}}+ {B(v) \wedge {A(w)}}).$$}

Denote the space of self-adjoint linear endomorphisms of {\small $\mathbb{R}^n$} and {\small $\wedge^2\mathbb{R}^n$} by {\small $S^2(\mathbb{R}^n)$} and {\small $S^2(\wedge^2\mathbb{R}^n)$} respectively. {\small$S^2_0(\mathbb{R}^n)$} denotes the space of traceless self-adjoint linear endomorphisms of {\small $\mathbb{R}^n$}. Following Hamilton ~\cite{H2} we define the sharp operator {\small$\#: S^2(\wedge^2\mathbb{R}^n) \times S^2(\wedge^2\mathbb{R}^n) \rightarrow S^2(\wedge^2\mathbb{R}^n),$} a bilinear map, by
{\small $$\langle (A \# B)(\phi),{\psi}\rangle = {\frac{1}{2}{\sum_{\alpha,\beta}{{\langle[A(\omega_{\alpha}),B(\omega_{\beta})],\phi\rangle}.{\langle[\omega_{\alpha},\omega_{\beta}],{\psi}\rangle}}}}$$}
\noindent
where {\small $A,B \in S^2(\wedge^2\mathbb{R}^n),$} {\small$\phi,\psi \in {\wedge^2\mathbb{R}^n}$} and
{\small$\{{\omega_{\alpha}}\}$} is an orthonormal basis of {\small$\wedge^2\mathbb{R}^n$}. It follows that {\small$A \# B = B \# A.$} Following Hamilton, we denote {\small $A \# A$ by $A^{\#}$}.

We define the bilinear operator {\small $Q : S^2(\wedge^2\mathbb{R}^n) \times S^2(\wedge^2\mathbb{R}^n) \rightarrow S^2(\wedge^2\mathbb{R}^n)$} by
$$Q(A,B) = \frac{1}{2}(AB+BA) +A \# B$$

and denote $Q(A,A)= A^2 + A^{\#}= Q(A)$.
Following Huisken ~\cite{Hu} we define the trilinear form  $\textit{tri}$ on {\small $S^2(\wedge^2\mathbb{R}^n)$} by
{\small $$\textit{tri}(A,B,C) = tr((AB+BA+2A\#B)\circ C)= 2tr(Q(A,B).C) = 2\langle Q(A,B), C \rangle .$$}

\noindent
Using the symmetry of $\#$, it can be checked that $\textit{tri}$ is symmetric in all three components.

\begin{defn}
The vector space of \textit{algebraic curvature operators} is defined as follows
$$S^2_{B}(\wedge^2\mathbb{R}^n) = \{R \in S^2(\wedge^2\mathbb{R}^n) \; : \mbox{ R satisfies the Bianchi identity}\}$$
\noindent
where {\small $R \in S^2(\wedge^2\mathbb{R}^n)$} is said to satisfy the Bianchi identity if
{\small $$\langle R(x \wedge y),z \wedge w\rangle + \langle R(y \wedge z),x \wedge w\rangle + \langle R(z \wedge x), y \wedge w\rangle =0, \ \ \forall \ \ x,y,z,w \in \mathbb{R}^n.$$}
\end{defn}
Let {\small$Ric(R): \mathbb{R}^n \rightarrow \mathbb{R}^n$} denote the \textit{Ricci tensor} of {\small$R \in S^2_{B}(\wedge^2\mathbb{R}^n)$} defined by
{\small $$Ric(R)_{ij} = \langle Ric(R)(e_i),e_j\rangle = \sum^n_{k = 1}{\langle R(e_i \wedge e_k),{e_j \wedge e_k} \rangle}$$}
\noindent
for an orthonormal basis {\small$\{e_i\}^n_{i=1}$} of {\small $\mathbb{R}^n$}.
Let {\small $Ric_0(R)$} denote the traceless part of {\small$Ric(R)$} and {\small$s(R)= tr(Ric(R))$} the \textit{scalar curvature} of $R.$

{\small$S^2_{B}(\wedge^2\mathbb{R}^n)$} has the following decomposition into three $O(n)$ invariant, irreducible and pairwise inequivalent subspaces
{\small $$S^2_{B}(\wedge^2\mathbb{R}^n) = \langle {\textbf{I}} \rangle \oplus \langle {\textbf{Ric}_0} \rangle \oplus \langle {\textbf{W}}\rangle$$}
where {\small $\langle {\textbf{I}} \rangle = \mathbb{R}(id \wedge id)$} denotes the scalar multiples of the identity operator on $\wedge^2\mathbb{R}^n$;
{\small $\langle {\textbf{Ric}_0} \rangle = id \wedge S^2_0(\mathbb{R}^n)$} denotes the space of algebraic curvature operators with traceless Ricci tensor  and
{\small $\langle{\textbf{W}}\rangle = \ker(Ric)$}  denotes the space of Weyl curvature operators i.e., algebraic curvature operators with vanishing Ricci tensor.

Then any algebraic curvature operator $R$ can be uniquely expressed as
{\small $$R = R_{\textbf{I}} + R_{\textbf{Ric}_0} + R_{\textbf{W}},$$}
where {\small$R_{\textbf{I}} , R_{\textbf{Ric}_0} , R_{\textbf{W}}$} denote the projections of $R$ onto corresponding irreducible components of {\small $S^2_B(\wedge^2\mathbb{R}^n)$}.
It follows that, {\small $$R_{\textbf{I}} = \frac{s(R)}{n(n-1)}id \wedge id \ \ {\rm and} \ \ R_{\textbf{Ric}_0} = \frac{2}{n-2}Ric_0(R) \wedge id.$$}
\begin{defn}\label{algebra}
We say that an algebraic curvature operator $R$ is

(a). Of \textit{Ricci type}, if {\small$R_{\textbf{W}} = 0;$}

(b). \textit{Einstein}, if {\small$R_{\textbf{Ric}_0} = 0$} i.e., if {\small $Ric(R)$} is a scalar multiple of {\small $id_{\mathbb{R}^n}.$} An Einstein curvature operator $R$ is called \textit{Ricci-flat}, if {\small$R = R_{\textbf{W}}.$}
\end{defn}

\section{The vector fields $Q$ and $\widetilde{Q}$}

It it easy to see that for {\small $R \in S^2_B(\wedge^2\mathbb{R}^n)$}, one has {\small$Q(R) \in S^2_B(\wedge^2\mathbb{R}^n)$} i.e., {\small$Q : R \mapsto R^2+R^{\#}$} defines a vector field on {\small$S^2_B(\wedge^2\mathbb{R}^n)$}.
Huisken ~\cite{Hu} observed that $Q$ is a gradient field i.e., $Q = \nabla{P}$ where $P$ is the smooth real valued function on {\small$S^2_{B}(\wedge^2\mathbb{R}^n)$} defined by
{\small $$P(R) = \frac{1}{6}{\textit{tri}(R,R,R)} = \frac{1}{3}tr(R^3 + RR^{\#})= \frac{1}{3}\langle Q(R), R \rangle.$$}
The Ricci tensor and scalar curvature of {\small $Q(R)$} are the following (see ~\cite{H1}, ~\cite{H2})
{\small \begin{equation}\label{scal}
\begin{cases}
Ric(Q(R))_{ij} = \sum_{k,l} Ric(R)_{kl}R_{ikjl}\\
s(Q(R)) = \|Ric(R)\|^2
\end{cases}
\end{equation}}
\noindent
where $\{e_i\}$ is an orthonormal basis of {\small $\mathbb{R}^n$} and {\small $R_{ikjl} = \langle R(e_i \wedge e_k) , (e_j \wedge e_l)\rangle$}. We will interchangeably use the notation $Q$ to denote the vector field and the bilinear operator mentioned earlier.

Using \eqref{scal} it follows that {\small $\widetilde{Q}$} defined earlier is a vector field on the hyperplane {\small $\mathbf{S}_1$}
(defined in Section 2), since {\small $s(\widetilde{Q}(R))=0$} when {\small $R \in \mathbf{S}_1$}. However it can be checked that the modified vector field {\small$\widetilde{Q}$}, is not a gradient field.

\begin{rem}\label{repara}
One observes that in order to infer the stable (respectively unstable) behavior of the original \eqref{reaone} near an algebraic curvature operator $R_0,$ it suffices to prove that $R_0$ is a stable (respectively unstable) zero of \eqref{reatwo}
as the modified ODE \eqref{reatwo} is obtained by normalizing the ODE \eqref{reaone} by its scalar curvature and reparametrizing it as follows.

Let $R(t): t \in [0,T]$ be a solution curve of the ODE \eqref{reaone} such that $s(R(t))>0,\ \  \forall t \in [0,T].$ Then one can normalize the curve $R(t)$ by its scalar curvature $s(t)=s(R(t))$ and the normalized curve $\bar{R}(t)= \frac{R(t)}{s(t)}$ satisfies the equation
\begin{equation}\label{scalone}
\frac{d\bar{R}}{dt} = s(R)\big(Q(\bar{R}) -\|Ric (\bar{R})\|^2 \bar{R}\big)
\end{equation}
Define $\varphi: [0,T] \rightarrow [0,S]$ by $$\varphi(t)= \int_0^t s(R(\mu))d\mu$$
which is an increasing function with $S= \varphi(T)$
Then $\widetilde{R}(\tau)= (\bar{R}\circ \varphi^{-1})(\tau), \tau \in (0,S]$ is a $1$-parameter family of curvature operators satisfying the equation \begin{equation*}
\begin{cases}
\frac{d\widetilde{R}}{d\tau} = Q(\widetilde{R}) -\|Ric (\widetilde{R})\|^2 \widetilde{R}\\
s(\widetilde{R})=1
\end{cases}
\end{equation*}
\end{rem}
Hence studying the stability (instability) of a curvature operator with respect to reaction ODE \eqref{reaone} is equivalent to studying the same with respect the modified ODE \eqref{reatwo}.\\

We now discuss some useful properties of $Q.$ First recall the definition of the restricted holonomy group {\small $Hol^0_p$} and the holonomy Lie algebra $\mathfrak{hol}_p$ at each point $p$ on a Riemannian manifold $(M,g)$.
Also recall the following classical proposition.
\begin{prop}
Given $v,w \in T_pM,$ the curvature transformation {\small $R(v,w): T_pM \rightarrow T_pM$} defined by $$R(v,w)(x)= R(v,w)x \ \ \forall \ \ x\in T_pM $$
is a skew symmetric linear operator where $R(v,w)x$ is the $(1,3)$ curvature tensor of $M.$ Then for each pair $v,w \in T_pM,$ the image of $R(v,w)$ lies in $\mathfrak{hol}_p.$
If $(M,g)$ is locally symmetric, then $\mathfrak{hol}_p$ is generated by the curvature transformations $R(v,w)$ as above.
\end{prop}
In light of the above proposition, we define the algebraic analogue of the holonomy algebra as follows. This definition suffices for our purpose as we mainly consider curvature operators of (locally) symmetric spaces.
\begin{defn}
The \textit{holonomy algebra} of an algebraic curvature operator {\small $R \in S^2_B(\wedge^2\mathbb{R}^n),$} denoted by $\mathfrak{hol}(R)$, is defined as the smallest Lie subalgebra of $\mathfrak{so}(n)$ containing the image of $R.$
\end{defn}
Recall that if $(M,g)$ is a Riemannian manifold such that at each $p \in M,$ one has {\small $Hol^0_p= G$} and if $g(t)$ be the solution of the Ricci flow \eqref{flow} starting at $g,$ then {\small $Hol^0_p(g(t)) \subset G$} for all time $t >0$ ~\cite{H2}. We prove the ODE version of this fact as follows:
\begin{prop}\label{pro}
Let $R$ be an algebraic curvature operator. Then the vector field $Q$ preserves holonomy i.e., {\small $\mathfrak{hol}(Q(R)) \subset \mathfrak{hol}(R).$}
\end{prop}
\begin{proof}
It suffices to show that {\small $Q(R)(\rho) \in \mathfrak{hol}(R)$} for any {\small $\rho \in \mathfrak{so}(n).$}
Let {\small $\{\xi_j\}_{j=1}^N$} be an orthonormal basis of $\mathfrak{so}(n)$ such that {\small $\{\xi_j\}_{j=1}^m$} spans $\mathfrak{hol}(R).$ Here $N$ and $m$ denote the respective dimensions of $\mathfrak{so}(n)$ and $\mathfrak{hol}(R)$. Given any {\small$\rho \in \mathfrak{so}(n)$} and {\small$\eta \in\mathfrak{hol}(R)^{\bot} \subset \mathfrak{so}(n)$} one notes that
{\small \begin{align*}
\langle Q(R)(\rho), \eta \rangle &= \langle Q(R)(\eta), \rho \rangle\\
&=\frac{1}{2}\sum_{i,j \leq N} \langle [R(\xi_i), R(\xi_j)], \eta \rangle \langle [\xi_i, \xi_j],\rho \rangle\\
&=\frac{1}{2}\sum_{i,j,k,l \leq N} \langle R(\xi_i),\xi_k \rangle \langle R(\xi_j),\xi_l \rangle\langle [\xi_k, \xi_l],\eta \rangle \langle [\xi_i, \xi_j],\rho \rangle\\
&= 0
\end{align*}}
as all surviving terms must have {\small $k,l \leq m$} and then {\small $\xi_k, [\xi_k, \xi_l] \in \mathfrak{hol}(R)$} which gives
{\small $\langle [\xi_k, \xi_l],\eta \rangle =0.$}  This shows that {\small $Q(R)(\rho) \in \mathfrak{hol}(R).$}
\end{proof}

We now focus on a special class of curvature operators:

\begin{defn}
{\small $R \in S^2_B(\wedge^2\mathbb{R}^{n})$} is said to be a \textit{product curvature operator} if there exist positive integers {\small$n_1, n_2$} with {\small $n=n_1+n_2$} and curvature operators
{\small $R_i \in S^2_B(\wedge^2\mathbb{R}^{n_i})$}, $i=1,2$ such that {\small$R= R_1+R_2$} i.e., {\small$R(v \wedge w) = R_i(v \wedge w)$} if {\small$v,w \in \mathbb{R}^{n_i}$} and {\small$R(v \wedge w) =0$} if {\small$v \in \mathbb{R}^{n_1}$}, {\small$w \in \mathbb{R}^{n_2}$}.
\end{defn}

We then have the following simple proposition:
\begin{prop}\label{new}
If {\small $R \in S^2_B(\wedge^2\mathbb{R}^{n_1 + n_2})$} is a product curvature operator then {\small $$\mathfrak{hol}(R) \in \mathfrak{so}(n_1) \oplus \mathfrak{so}(n_2).$$}
\end{prop}
The proof is immediate from the definition of a product curvature operator.
An easy consequence of the above propositions is the following
\begin{cor}\label{coro}
$Q$ preserves the product structure i.e., if {\small$R =R_1 + R_2 \in S^2_B(\wedge^2\mathbb{R}^{n_1+n_2})$} is a product curvature operator, then {\small$Q(R)=Q(R_1) + Q(R_2)$} is also a product of $Q(R_1)$ and $Q(R_2)$ as defined above.
\end{cor}

The following Proposition describes few more well known and imprortant properties of $Q$.
\begin{prop}\label{propn}
(1). If {\small $R \in S^2_B(\wedge^2\mathbb{R}^{n})$} is an Einstein curvature operator, then so is $Q(R).$ Moreover, If $R$ is Ricci-flat, then so is $Q(R).$\\

(2). Given {\small $R \in S^2_B(\wedge^2\mathbb{R}^n),$} one has
{\small $$Q(R,\mathbf{I})= R + R \# \mathbf{I} =(n-1)R_{\textbf{I}} + \frac{(n-2)}{2}R_{\textbf{Ric}_0}.$$}
In particular, {\small $Q(\mathbf{I},W) =0$} if {\small $W \in \langle {\textbf{W}}\rangle.$}\\

(3).  For {\small $R \in \langle {\textbf{Ric}_0} \rangle$} and {\small $S,W \in \langle {\textbf{W}}\rangle;$} one has
{\small $$\textit{tri}(R,S,W) = 0 = \textit{tri}(S,R,\textbf{I}) \ \ {\rm and} \ Q(S,R) \in \langle {\textbf{Ric}_0} \rangle.$$}

(4). If {\small $R \in  S^2_B(\wedge^2\mathbb{R}^n)$} is a curvature operator of Ricci type, then one has
{\small \begin{align}
\nonumber Q(R) =& \frac{1}{n-2}Ric_0(R)\wedge Ric_0(R) + \frac{2s(R)}{n(n-1)}Ric_0(R)\wedge id_{\mathbb{R}^n}\\
 & -\frac{2}{(n-2)^2}\big(Ric_0(R)^2\big)_0 \wedge id_{\mathbb{R}^n} + \big(\frac{s(R)^2}{n^2(n-1)}+ \frac{\|Ric_0(R)\|^2}{n(n-2)} \big)\mathbf{I}
\end{align}}
In particular, if {\small $R \in \langle {\textbf{Ric}_0} \rangle,$} then
{\small \begin{equation}
Q(R) = \frac{1}{n-2}Ric_0(R)\wedge Ric_0(R)_{\mathbf{W}} -\frac{4}{(n-2)^2}\big(Ric_0(R)^2\big)_0 \wedge id_{\mathbb{R}^n} + \frac{\|Ric_0(R)\|^2}{n(n-1)} \mathbf{I}
\end{equation}}
\end{prop}

\begin{proof}
The first part follows immediately from \eqref{scal}.
The second part follows from ~\cite{BW} (Lemma 2.1).\\

The third part can be proved using part $1$ and observing that for {\small $S,W \in \langle {\textbf{W}}\rangle;$}  {\small $Q(S,W) \in \langle {\textbf{W}}\rangle.$} Then for {\small $R \in \langle {\textbf{Ric}_0} \rangle,$}
{\small $$\textit{tri}(R,S,W) = \textit{tri}(W,S,R) = 2\langle Q(W,S), R\rangle =0.$$}

Using part $2$, it follows that {\small $Q(S,\mathbf{I})= S + S \# \mathbf{I} =0$} for {\small $S \in \langle {\textbf{W}}\rangle.$} This shows that
{\small $$\textit{tri}(S,R,\textbf{I}) = \textit{tri}(S,\textbf{I},R) = 2\langle Q(S,\textbf{I}), R\rangle =0.$$}
The last part of $3$ follows immediately from the rest.\\

For the fourth part see ~\cite{BW} (Lemma 2.2).
\end{proof}

\section{Zeros of $Q$ and $\widetilde{Q}$}

In this section we describe the zeros of $Q$ and $\widetilde{Q}$ that are curvature operators of symmetric spaces and also prove Theorems \ref{th1} and \ref{th2}. Note that if {\small  $Q(R)=0,$} then using \eqref{scal} one obtains {\small $s(Q(R))= \|Ric(R)\|^2 =0$} i.e., $R$ is Ricci flat. One further observes that
\begin{prop}\label{symm}
Let $(M^n, g)$ be a closed Einstein manifold with unit ( respectively zero) scalar curvature. Then the Riemannian curvature operator $\mathfrak{R}$ is a zero of $\widetilde{Q}$ (respectively $Q$) at each point in $M$ if and only if $(M, g)$ is locally symmetric (respectively flat).
\end{prop}
\begin{proof}
Let $(M^n, g)$ be Einstein with unit scalar curvature. Then the $(0, 4)$ curvature tensor $R_g$ satisfies the following identity (see ~\cite{PT}, page no. 40)
\begin{equation}\label{laplacee}
\Delta R_g + 2Q(R_g) = \frac{2}{n}R_g
\end{equation}
\noindent
where $Q(R_g)$ is the $(0,4)$ tensor which induces the self-adjoint linear operator $Q(\mathfrak{R})_p$ on $\wedge^2T_pM$ for each $p \in M.$
Again due to unit scalar curvature, we have $\|Ric_g\|^2 = \frac{1}{n}$
and by our assumption, $$Q(R_g) = \|Ric_g\|^2R_g = \frac{1}{n}R_g$$ at each point in $M$.

Combining these two it follows that at each point in $M,$  $\Delta R_g = 0$. Since $(M, g)$ is closed, it is equivalent to saying that $R_g$ is parallel and $(M,g)$ is locally symmetric.

The converse is easy to see as the curvature operator of every locally symmetric Einstein manifold with unit scalar curvature is a zero of $\widetilde{Q}$ by the above identity ~\eqref{laplacee}.

For the Ricci flat case the proof is similar.
\end{proof}

Next one observes that
\begin{prop}\label{prop}
Let {\small $(M_i^{n_i}, g_i)$} ($i=1,2$) be symmetric Einstein manifolds such that the curvature operator of the product space {\small $(M_1^{n_1} \times M_2^{n_2}, \tilde{g})$} is a zero of $\widetilde{Q}$ where $\tilde{g}$ is the product metric normalized to unit scalar curvature. Then either $\tilde{g}$ is Einstein or one of $g_1,g_2$ is flat.
\end{prop}

\begin{proof}
Let $\lambda_i$ denote the Einstein constants of $g_i\;, i = 1, 2.$  Then {\small $\tilde{g} = (n_1\lambda_1 + n_2\lambda_2)(g_1 + g_2)$} and at least one of the constants is nonzero as {\small $s(\tilde{g}) = 1$}.

Consider an orthonormal basis {\small $\{e_i\}^{n_1 + n_2}_{i = 1}$} at a point {\small $(m_1, m_2)\in M_1 \times M_2$} where {\small $\{e_i\}^{n_1}_{i = 1}$} and {\small $\{e_i\}^{n_1 + n_2}_{i = n_1 + 1}$} form bases of {\small $T_{m_1}M^{n_1}_1$} and {\small $T_{m_2}M^{n_2}_2$} respectively. Let {\small $R_{i},\; Ric_{i},\; \mathfrak{R}_i$} denote the curvature tensor, Ricci tensor and curvature operator of {\small $g_i, \; i = 1, 2$}. Let {\small $R_{\tilde{g}},\; Ric_{\tilde{g}},\; \mathfrak{R}$} denote corresponding entities of $\tilde{g}$.
We have
{\small \begin{equation}\label{prop1}
(Ric_{\tilde{g}})_{ii} = \begin{cases} (Ric_{1})_{ii}\; =\; \frac{\lambda_1}{n_1\lambda_1 + n_2\lambda_2} \;\;\mbox{if}\;\;\; 1 \leq i \leq n_1; \\
\    \\(Ric_{2})_{ii}\; =\; \frac{\lambda_2}{n_1\lambda_1 + n_2\lambda_2} \;\;\mbox{if}\;\;\; n_1 + 1 \leq i \leq n_1 + n_2.
\end{cases}
\end{equation}}
and {\small $\|Ric_{\tilde{g}}\|^2 = \frac{n_1\lambda_1^2 + n_2\lambda_2^2}{(n_1\lambda_1 + n_2\lambda_2)^2}$}.

The Laplacian $\Delta R_{\tilde{g}}$ is given by  (see ~\cite{PT}, page no. 40)
{\small \begin{equation}\label{lap}
\begin{split}
&\big(\Delta R_{\tilde{g}} + 2Q(R_{\tilde{g}})\big)(x, y, w, z)\\ &= \nabla^2_{x,w}Ric_{\tilde{g}}(y,z)+ \nabla^2_{y,z}Ric_{\tilde{g}}(x,w)- \nabla^2_{x,z}Ric_{\tilde{g}}(y,w)-\nabla^2_{y,w}Ric_{\tilde{g}}(x,z)\\
&+ Ric_{\tilde{g}}(R_{\tilde{g}}(w, z )x, y) - Ric_{\tilde{g}}(R_{\tilde{g}}(w, z )y, x)
\end{split}
\end{equation}}
where {\small $Q(R_{\tilde{g}})(x, y, w, z) = Q(\mathfrak{R}_{\tilde{g}})(x, y, w, z)$}.

As {\small $R_{\tilde{g}}$} is parallel, using \eqref{lap} and \eqref{prop1} one obtains
{\small \begin{equation*}
Q(R_{\tilde{g}})_{ijkl} = \begin{cases}\frac{\lambda_1}{n_1\lambda_1 + n_2\lambda_2} (R_{1})_{ijkl} \;\;\mbox{if}\;\;\; 1 \leq i, j, k, l \leq n_1;\\
\             \\\frac{\lambda_2}{n_1\lambda_1 + n_2\lambda_2} (R_{2})_{ijkl} \;\;\mbox{if}\;\;\; n_1 + 1 \leq i, j, k, l \leq n_1 + n_2;\\
\     \\\;\;\;0\;\;\;\;\;\mbox{otherwise}.
\end{cases}
\end{equation*}}
Let both the constants be non zero.
Then {\small $Q(R_{\tilde{g}}) = \|Ric_{\tilde{g}}\|^2R_{\tilde{g}}$} gives
{\small \begin{equation}\label{product}
\begin{cases}
\bigg(\frac{\lambda_1}{n_1\lambda_1 + n_2\lambda_2} - \frac{n_1\lambda_1^2 + n_2\lambda_2^2}{(n_1\lambda_1 + n_2\lambda_2)^2}\bigg)(R_{1})_{ijkl} = 0 \;\;\mbox{if}\;\;i, j, k, l \leq n_1;\\
\                            \\\bigg(\frac{\lambda_2}{n_1\lambda_1 + n_2\lambda_2} - \frac{n_1\lambda_1^2 + n_2\lambda_2^2}{(n_1\lambda_1 + n_2\lambda_2)^2}\bigg)(R_{2})_{ijkl} = 0 \;\;\mbox{if}\;\;i, j, k, l \geq (n_1 + 1).
\end{cases}
\end{equation}}
As none of $g_1, g_2$ is flat, this means
{\small $$\frac{n_1\lambda_1(\lambda_1 - \lambda_2)}{(n_1\lambda_1 + n_2\lambda_2)^2} = \frac{n_2\lambda_2(\lambda_2 - \lambda_1)}{(n_1\lambda_1 + n_2\lambda_2)^2} = 0$$}
i.e.,$\lambda_1 = \lambda_2$ and $\tilde{g}$ is Einstein. \\

If one of the constants is zero then the corresponding metric is flat. For definiteness, let {\small $\lambda_2=0.$} Then $g_1$ is Einstein with unit scalar curvature i.e., {\small $\lambda_1= \frac{1}{n_1}$} and one has {\small $(R_{2})_{ijkl} = 0 \;\;\forall \;\;i, j, k, l \geq (n_1 + 1).$} Using \eqref{product} we see that {\small $Q(R_{\tilde{g}}) = \|Ric_{\tilde{g}}\|^2R_{\tilde{g}}$} still holds. This completes the proof.
\end{proof}

\begin{rem}
Note that if $(M^n, g)$ is not a spherical space form or $k >1$, then the product curvature operator of {\small $M^n \times \mathbb{R}^k$} is a zero of $\widetilde{Q}$ which is neither Einstein and nor of Ricci type.

Since a symmetric space can be decomposed into Riemannian product of irreducible symmetric spaces and Euclidean space, it follows that Proposition \ref{prop} describes all the zeros of $\widetilde{Q}$ corresponding to curvature operators of (locally) symmetric spaces.
\end{rem}

The question still remains whether there exist algebraic curvature operators that are zeros of $Q$ or $\widetilde{Q}$ which can not be realized as the Riemann curvature operator of any (locally) symmetric space. Where we do not know the complete set of zeros of $Q$ or $\widetilde{Q}$ in full generality, we give a full description of the zeros of $Q$ and Einstein zeros of $\widetilde{Q}$ in dimension $4$ and also zeros of $\widetilde{Q}$ of Ricci type in all dimensions as follows.
First we prove that
\begin{prop}\label{unique}
Let {\small $R \in S^2_B(\wedge^2{\mathbb{R}^4})$} be a zero of $Q.$ Then $R$ is the zero operator.
\end{prop}
\begin{proof}
The Hodge star operator {\small $\ast: \wedge^2{\mathbb{R}^4} \rightarrow \wedge^2{\mathbb{R}^4}$} defined by
{\small $$\langle \ast\eta_1 , \eta_2\rangle\theta = \eta_1 \wedge \eta_2; \ \ \eta_1, \eta_2 \in \wedge^2{\mathbb{R}^4}$$}
\noindent
where {\small $\theta \in \wedge^4 \mathbb{R}^4$ denotes an orientation on {\small $\mathbb{R}^4,$} is an involution and has eigenvalues $\pm1$. One has the following orthogonal decomposition {\small $\wedge^2{\mathbb{R}^4} = \wedge^{+}\mathbb{R}^4 \oplus \wedge^{-}\mathbb{R}^4$}}
where  {\small $\wedge^{\pm}\mathbb{R}^4$} denote $`\pm'$ eigen spaces of $\ast$ respectively.

Let {\small $R \in S^2_B(\wedge^2{\mathbb{R}^4})$} be Einstein i.e., {\small $R = \frac{s(R)}{12}\mathbf{I} + R_{\textbf{W}}.$} Then $R$ preserves the above orthogonal decomposition i.e.,
{\small $R(\wedge^{+}\mathbb{R}^4) \subset \wedge^{+}\mathbb{R}^4 \ \ {\rm and} \ \ R(\wedge^{-}\mathbb{R}^4)\subset \wedge^{-}\mathbb{R}^4$}
and the matrix of $R$ with respect to a basis of {\small $\wedge^2{\mathbb{R}^4}$} respecting the orthogonal decomposition, is of the form
{\small $$R=\left(
  \begin{array}{cc}
    \frac{s(R)}{12}\mathbf{I} + R^{+}_{\textbf{W}} & \mathbf{0} \\
    \mathbf{0} & \frac{s(R)}{12}\mathbf{I} + R^{-}_{\textbf{W}} \\
  \end{array}
\right)$$}
\noindent
where {\small $R^{+}_{\textbf{W}} = R_{\textbf{W}}|_{\wedge^{+}\mathbb{R}^4}$} and {\small $R^{-}_{\textbf{W}} = R_{\textbf{W}}|_{\wedge^{-}\mathbb{R}^4}$} denote the self-dual and anti self-dual Weyl components of $R$ and $\mathbf{0}$ the zero matrix. Also {\small $tr(R^{\pm}_{\textbf{W}}) =0.$}

If {\small $Q(R)= \mathbf{0}$} then {\small $R = R_{\textbf{W}}$} with {\small $R^2_{\textbf{W}}+R^{\#}_{\textbf{W}} = 0$} and the matrix of $R$ reduces to
{\small $$R= \left(
  \begin{array}{cc}
    R^{+}_{\textbf{W}} & \mathbf{0} \\
    \mathbf{0} & R^{-}_{\textbf{W}} \\
  \end{array}
\right)$$}
Let {\small $\{ \xi_1,....,\xi_6\}$} be an orthonormal basis of {\small $\wedge^2{\mathbb{R}^4}$} diagonalizing $R.$ Without loss of generality, assume that
{\small $\{\xi_1, \xi_2, \xi_3\} \;$} and {\small $\; \{\xi_4, \xi_5, \xi_6\}\;$} span {\small $\wedge^{+}\mathbb{R}^4$} and {\small $\wedge^{-}\mathbb{R}^4$} respectively. Then matrices of {\small $R^{\pm}_{\textbf{W}}$} with respect to these ordered bases are diagonal and given by
{\small $$R^{+}_{\textbf{W}} = diag(x_1,x_2,x_3), \ \ R^{-}_{\textbf{W}}=diag(y_1,y_2,y_3)$$}
where {\small $x_1 + x_2 + x_3= 0= y_1 + y_2 + y_3.$} Here by $diag(a,b,c)$ we denote the diagonal matrix of order $3$ with diagonal entries $a,b,c$ respectively. \\

From Proposition \ref{propn}, it follows that $Q(R)$ is Ricci flat. It is easy to see that {\small $Q(R_{\textbf{W}})^{\pm} = Q(R^{\pm}_{\textbf{W}})$} and the matrices of {\small $Q(R^{\pm}_{\textbf{W}})$} with respect to the same ordered bases are also diagonal of the form
{\small $$Q(R^{+}_{\textbf{W}})=diag(x_1^2 + c_1 x_2x_3,x_2^2 + c_1 x_1x_3,x_3^2 + c_1 x_1x_2)$$}
and
{\small $$Q(R^{-}_{\textbf{W}})=diag(y_1^2 + c_2 y_1y_2, y_2^2 + c_2 y_2y_3, y_3^2 + c_2 y_1y_2)$$}
where {\small $c_1= \langle [\xi_1, \xi_2], \xi_3\rangle^2 $} and {\small $c_2= \langle [\xi_4, \xi_5], \xi_6\rangle^2$} are square of the non-vanishing structure constants of the above basis. One has {\small $Q(R^{\pm}_{\textbf{W}})= \mathbf{0}$} which gives the following systems of homogeneous equations
\begin{align}
\begin{cases}
x_1^2 + c_1 x_2x_3 = 0 ;\\
x_2^2 + c_1 x_1x_3 = 0 ;\\
x_3^2 + c_1 x_1x_2 =0
\end{cases}
\end{align}
\begin{align}
\begin{cases}
y_1^2 + c_2 y_2y_3 = 0 ;\\
y_2^2 + c_2 y_1y_3 = 0 ;\\
y_3^2 + c_2 y_1y_2 =0
\end{cases}
\end{align}
with {\small $x_1 + x_2 + x_3 = 0 = y_1 + y_2 + y_3.$}
It follows that the systems admit only trivial solution i.e., {\small $R^{\pm}_{\textbf{W}} = \mathbf{0}.$}
\end{proof}
The above proof requires the orthogonal splitting of {\small $\wedge^2\mathbb{R}^4$} which is not true for dimensions more than $4$. Unlike dimension $4$, in higher dimensions an eigen basis of an Einstein curvature operator $R$ need not always diagonalize $Q(R)$ and one fails to get such a simplified expression for $Q(R).$

For Einstein zeros of $\widetilde{Q}$ in dimension $4$ we prove the following
\begin{prop}\label{uniquea}
If {\small $R \in S^2_B(\wedge^2{\mathbb{R}^4})$} is an Einstein zero of $\widetilde{Q}$, then up to the action of the orthogonal group, $R$ is one of the following
\begin{enumerate}
\item  Curvature operator of the round metric on $S^4$ normalized to unit scalar curvature,
\item Curvature operator of the standard product metric on $S^2 \times S^2$ normalized to unit scalar curvature,
\item Curvature operator of the Fubini-Study metric on $\mathbb{C}\mathbb{P}^2$ normalized to unit scalar curvature.
\end{enumerate}
\end{prop}
\begin{proof}
Let {\small $R \in S^2_B(\wedge^2{\mathbb{R}^4})$} be an Einstein zero of $\widetilde{Q}$  i.e., $s(R) = 1$ and {\small $$R^2+R^{\#}= \|Ric(R)\|^2R = \frac{1}{4}R.$$}
It suffices to show that $R$ has same eigenvalues as one of the curvature operators mentioned in the hypothesis.
Following Proposition \ref{unique}, the matrix of $R$ with respect to a basis of {\small $\wedge^2{\mathbb{R}^4}$} which respects the orthogonal decomposition {\small $\wedge^2{\mathbb{R}^4} = \wedge^{+}\mathbb{R}^4 \oplus \wedge^{-}\mathbb{R}^4,$} is of the form
{\small $$R= \left(
  \begin{array}{cc}
   \frac{1}{12}\mathbf{I}+ R^{+}_{\textbf{W}} & \mathbf{0} \\
    \mathbf{0} & \frac{1}{12}\mathbf{I} + R^{-}_{\textbf{W}} \\
  \end{array}
\right)$$}
\noindent
where {\small $R^{+}_{\textbf{W}}$} and {\small $R^{-}_{\textbf{W}}$} denote the self-dual and anti self-dual Weyl components of $R$ respectively. Also {\small $tr(R^{\pm}_{\textbf{W}}) =0.$}
Consider an ordered orthonormal basis {\small $\{ \xi_1,....,\xi_6\}$} of {\small $\wedge^2{\mathbb{R}^4}$} as in Proposition \ref{unique} which diagonalizes $R.$ Denoting {\small $\frac{1}{12}\mathbf{I}+R^{\pm}_{\textbf{W}}$} by {\small $R^{\pm}$}, one has
{\small $$R^{+} = diag(x_1,x_2,x_3) \ \ {\rm and} \ \ R^{-}=diag(y_1,y_2,y_3)$$}
with  {\small $x_1 + x_2 + x_3= \frac{1}{4}= y_1 + y_2 + y_3,$} as
{\small $tr( R^{\pm}) = \frac{1}{4}.$}

Eigenvalues of $R^{\pm}$ collectively give the eigenvalues of $R.$  If $R$ is one of the Riemannian curvature operators mentioned in the hypothesis, then $R^{\pm}$ and their eigenvalues are listed below
\begin{enumerate}
\item  For curvature operator of the normalized round sphere $S^4$, {\small $R^{\pm}=\frac{1}{12}Id$} i.e., the only eigenvalue is {\small $\frac{1}{12}$} with multiplicity three;
\item For curvature operator of the normalized product $S^2 \times S^2$, $R^{\pm}$ have eigenvalues {\small $0,0,\frac{1}{4};$}
\item For curvature operator of the normalized Fubini-Study metric on $\mathbb{C}\mathbb{P}^2$, $R^{+}$ has eigenvalues {\small $0,0,\frac{1}{4}$} and {\small $R^{-}=\frac{1}{12}Id$}.
\end{enumerate}

Since $R$ is Einstein, so is $Q(R)$which is also diagonalized by the eigen basis of $R$ and {\small $Q(R^{\pm})= Q(R)^{\pm}= \frac{1}{4}R^{\pm}.$}

As in Proposition \ref{unique}, the matrices of {\small $Q(R^{\pm})$} with respect to the ordered bases are
{\small $$Q(R^{+})=diag(x_1^2 + c_1 x_2x_3,x_2^2 + c_1 x_1x_3,x_3^2 + c_1 x_1x_2)$$}
and
{\small $$Q(R^{-})=diag(y_1^2 + c_2 y_1y_2, y_2^2 + c_2 y_2y_3, y_3^2 + c_2 y_1y_2)$$}
where $c_1$ and $c_2$ are as in Proposition \ref{unique}.
Then it reduces to solving the system
{\small \begin{align}\label{sec}
\begin{cases}
\lambda_1^2 + \mu \lambda_2\lambda_3 &= \frac{1}{4}\lambda_1 ;\\
\lambda_2^2 + \mu \lambda_1\lambda_3 &= \frac{1}{4}\lambda_2 ;\\
\lambda_3^2 + \mu \lambda_1\lambda_2 &= \frac{1}{4}\lambda_3;\\
\lambda_1 + \lambda_2 + \lambda_3 &= \frac{1}{4}.
\end{cases}
\end{align}}
where $\mu$ is a positive number.
It follows that the only non trivial solutions to this system are
{\small $$(\lambda_1, \lambda_2, \lambda_3) = (0,0,\frac{1}{4}), (0,\frac{1}{4},0), (\frac{1}{4},0,0), (\frac{1}{12},\frac{1}{12},\frac{1}{12}).$$}
These give all possible eigenvalues of $R^{\pm}$ which coincide with those of the curvature operators enlisted in the hypothesis, and the proposition follows.
\end{proof}

The list of algebraic curvature operators given in Proposition \ref{uniquea} describes (up to scaling) all the curvature operators of (locally) symmetric Einstein $4$-manifolds with positive scalar curvature. One can thus say that in dimension $4$ an algebraic Einstein curvature operator $R$ with positive scalar curvature satisfying $Q(R) = \lambda R,$ must be the curvature operator of some (locally) symmetric $4$-manifold. Propositions \ref{unique} and \ref{uniquea} together complete the proof of Theorem \ref{th1}.

For dimensions greater than $4$ one can not say much about the full set of Einstein zeros of $\widetilde{Q}.$ Although for zeros of $\widetilde{Q}$ of Ricci type we prove that

\begin{thm}
Let {\small $R \in S^2_B(\wedge^2\mathbb{R}^{n+1})$} be a zero of $\widetilde{Q}$ of Ricci type which is not a multiple of the identity. Then up to the action of the orthogonal group, $R$ is the curvature operator of the standard product metric on {\small $S^n\times \mathbb{R}$} normalized to unit scalar curvature.
\end{thm}
\begin{proof}
It suffices to show that $R$ has the same eigenvalues as the curvature operator of the product metric on {\small $S^n \times \mathbb{R}$} normalized to unit scalar curvature. By assumption, $R$ is of the form {\small $R = \frac{1}{n(n+1)}{\textbf{I}} + R_{\textbf{Ric}_0}$} and {\small $R_{\textbf{Ric}_0} \neq 0$}.

Now {\small $R_{\textbf{Ric}_0} = \frac{2}{n-1}Ric_0(R)\wedge id.$} Injectivity of the map {\small $A \mapsto A\wedge id\;$} for a self adjoint endomorphism $A$ of $\mathbb{R}^{n+1}$, gives {\small $Ric_0(R) \neq 0$}. Using {\small $Q(R) = \|Ric(R)\|^2R$} and  Proposition \ref{propn}, one sees that
{\small \begin{align}\label{four}
\nonumber Q(R)& = \bigg(\frac{1}{(n+1)^2n} + \frac{\|Ric_0(R)\|^2}{(n+1)(n-1)}\bigg)\textbf{I}\\
\nonumber&+ \frac{2}{n(n+1)}(Ric_0\wedge id) - \frac{2}{(n-1)^2}(Ric_0(R)^2)_0 \wedge id +\frac{1}{n-1}(Ric_0(R)\wedge Ric_0(R))\\
\nonumber& = \bigg(\frac{1}{n(n+1)^2} + \frac{\|Ric_0(R)\|^2}{n(n+1)}\bigg) \textbf{I}
+ \frac{2}{n(n+1)}(Ric_0(R)\wedge id)\\
 &- \frac{4}{(n-1)^2}(Ric_0(R)^2)_0 \wedge id + \frac{1}{n-1}\big(Ric_0(R)\wedge Ric_0(R)\big)_{\textbf{W}}
\end{align}}
and
{\small \begin{align}\label{five}
\nonumber \|Ric(R)\|^2R& = \bigg(\frac{1}{n+1} + \|Ric_0(R)\|^2\bigg)\bigg(\frac{1}{n(n+1)}\textbf{I} + R_{\textbf{Ric}_0}\bigg)\\
\nonumber& =\bigg(\frac{1}{n(n+1)^2} + \frac{\|Ric_0(R)\|^2}{n(n+1)}\bigg)\textbf{I}\\
&+ \frac{2}{(n-1)}\bigg(\frac{1}{n+1} + \|Ric_0(R)\|^2\bigg)Ric_0(R)\wedge id
\end{align}}
Combining \eqref{four} and \eqref{five}, one obtains
{\small \begin{align} \label{six}
 \nonumber &(Ric_0(R)\wedge Ric_0(R))_{\textbf{W}} = 0 \ \ {\rm and}\\
& \bigg(\frac{1}{n(n-1)(n+1)} + \frac{\|Ric_0(R)\|^2}{n-1}\bigg)Ric_0(R)\wedge id + \frac{2}{(n-1)^2}(Ric_0(R)^2)_0 \wedge id =0
\end{align}}
Using the injectivity of {\small $A \mapsto A \wedge id$}, the last equation reduces to
{\small \begin{equation}\label{s}
\bigg(\frac{1 + n(n+1)\|Ric_0(R)\|^2}{n(n-1)(n+1)}\bigg)Ric_0(R) + \frac{2}{(n-1)^2}(Ric_0(R)^2)_0 = 0
\end{equation}}
Here $R$ is diagonalized by an orthonormal basis {\small $\{e_i\wedge e_j\}$} where {\small $\{e_i\}^{n+1}_{i=1}$} is an orthonormal basis of {\small $\mathbb{R}^{n+1}$} diagonalizing {\small $Ric_0(R)$}. Then it reduces to proving that {\small $Ric_0(R)$} and the traceless Ricci tensor of the normalized product metric on {\small $S^n \times \mathbb{R}$} have same eigenvalues.

Let $\lambda_i$ be the eigenvalue of {\small $Ric_0(R)$} corresponding to the eigenvector {\small $e_i$, $\,1\leq i\leq n+1$}. Then {\small $tr(Ric_0(R)) = \sum_i{\lambda_i} = 0$}.
As {\small $Ric_0(R) \neq 0,$} at least two of the $\lambda_i$s are not equal.
Then for each $i$, {\small $$Ric_0(R)^2(e_i) = \lambda^2_i e_i\,\; {\rm and}\;\, (Ric_0(R)^2)_0(e_i) =
\big(\lambda^2_i - \frac{1}{n+1}\Sigma_{j=1}^{n+1}{\lambda^2_j}\big)e_i.$$}
Also {\small $Ric_0(R)\wedge Ric_0(R)$} is diagonalized by {\small {\small $\{e_i\wedge e_j\}$} with {\small $$(Ric_0(R)\wedge Ric_0(R))(e_i \wedge e_j) = \lambda_i
\lambda_j e_i \wedge e_j.$$}
Using \eqref{s} one obtains
{\small \begin{equation}\label{new}
\bigg(\frac{1}{n(n+1)(n-1)} + \frac{1}{n-1}{\sum^{n+1}_{p=1}{\lambda^2_p}}\bigg)\lambda_i + \frac{2}{(n-1)^2}\bigg(\lambda^2_i -
\frac{1}{n+1}\sum^{n+1}_{p=1}{\lambda^2_p}\bigg) = 0,\,\,\forall \,i,\,\,\; 1\leq i\leq (n+1).
\end{equation}}
Using \eqref{six} one gets
{\small \begin{equation}\label{ric}
\langle Ric_0(R)\wedge Ric_0(R) , \xi \rangle = 0,\,\,\,\forall \,\xi \in \langle \textbf{W}\rangle.
\end{equation}}
Define {\small $\xi_1, \xi_2 \in S^2_B(\wedge^2\mathbb{R}^n)$} by
{\small $$\xi_1 = (A_1\wedge A_2 + A_3\wedge A_4 - A_1\wedge A_3 - A_2\wedge A_4) \;\,{\rm and} $$}
{\small $$\xi_2 =(A_1\wedge A_4 + A_2\wedge A_3 - \frac{1}{2}A_1\wedge A_2 - \frac{1}{2}A_1\wedge A_3 - \frac{1}{2}A_2\wedge A_4 - \frac{1}{2}A_3\wedge A_4)$$}
where $A_i$ is the self-adjoint endomorphism of $\mathbb{R}^{n+1}$ given by {\small $$A_i(e_j) = \delta_{ij}e_j$$} with $\{e_j\}$ as above. It follows that $Ric(\xi_1) = 0 = Ric(\xi_2)$ i.e., {\small $\xi_1,
\xi_2 \in \langle\textbf{W}\rangle$}.
Then \eqref{ric} gives
{\small $(\lambda_2 - \lambda_3)(\lambda_1 - \lambda_4) = 0$} and \\
{\small $\lambda_1\lambda_2 + \lambda_1\lambda_3 -2(\lambda_1\lambda_4 - \lambda_2\lambda_3) + \lambda_2\lambda_4 + \lambda_3\lambda_4 = 0$}.

Without loss of generality, assume {\small $\lambda_1 \neq \lambda_4$}} i.e., {\small $\lambda_2 = \lambda_3.$}
It follows from the last equation above that {\small $(\lambda_2 - \lambda_4)(\lambda_1 - \lambda_2) = 0$}.
For definiteness let {\small $\lambda_1 = \lambda_2$} i.e., {\small $\lambda_1 = \lambda_2 = \lambda_3 \neq \lambda_4$}.

Choosing {\small $\xi_p = A_1\wedge A_2 + A_p\wedge A_4 - A_1\wedge A_p - A_2\wedge A_4 \in \langle\textbf{W}\rangle$} for each $p \geq 5$ and proceeding
as before, one obtains
{\small $\lambda_i = \lambda\,\, \mbox{say},$} for $i\neq 4$. Consequently, {\small $\lambda_4 = -n\lambda$}.
Using \eqref{new} it follows that
{\small $\lambda = \frac{1}{n(n+1)}$}. Thus the eigenvalues of {\small $Ric_0(R)$} are precisely those of the traceless Ricci tensor of the normalized product
metric on {\small $S^n \times \mathbb{R}$} as discussed in Proposition \ref{ric1} in Section 5.
\end{proof}

\section{Stability of the reaction ODE}

In this section we analyze the stability of the ODE \eqref{reaone} near the curvature operators $\mathfrak{R}$ of symmetric spaces satisfying {small $Q(\mathfrak{R})=\lambda \mathfrak{R}$}, $\lambda >0$.  It follows directly from the following proposition which was suggested to us by Thomas Richard, that \eqref{reaone} is unstable near Einstein curvature operators $R$ that are not multiple of identity and satisfy $Q(R)= \lambda R.$
\begin{prop}\label{thomas}
Let $R^0$ be an algebraic curvature operator which is Einstein with positive scalar curvature but not a multiple of the identity and satisfy $Q(R^0) =\lambda R^0$ for some $\lambda >0.$ Then up to suitable normalization there is a solution curve to the ODE \eqref{reaone} emerging from a point near $R^0$ and approaching the (normalized) Weyl curvature of $R^0$.
\end{prop}
\begin{proof}
Observe that for an Einstein curvature operator $R^0$ as in the hypothesis, one has {\small$s(Q(R^0)) = \lambda s(R^0)$} i.e., {\small$\lambda = \frac{s(R^0)}{n}.$}
Also $R^0$ is of the form {\small$R^0 = R^0_{\textbf{I}} + R^0_{\textbf{W}}$} and it follows from the hypothesis and the properties of $Q(R)$ mentioned in Proposition \ref{propn} that
\begin{equation}\label{split}
Q(R^0_{\textbf{I}}) = \lambda R^0_{\textbf{I}}, \ \ Q(R^0_{\textbf{W}}) = \lambda R^0_{\textbf{W}}
\end{equation}

Let $R(t) = f(t)R^0$ be a solution to the ODE \eqref{reaone}. Then $f$ satisfies $f'(t) = \lambda f^2(t),$ as we have $$Q(R(t)) = Q(f(t)R^0)= f^2(t)Q(R^0)= \lambda f^2(t)R^0$$ and $\frac{dR}{dt} = f'(t)R^0.$

In fact, there is a unique solution $f(t)= \frac{1}{1-\lambda t}$ to the system
\begin{align}\label{function}
\begin{cases}
f'=\lambda f^2;\\
f(0)=1.
\end{cases}
\end{align}
Consider $\widetilde{R} = \alpha_0 R^0_{\textbf{I}}+ R^0_{\textbf{W}}$ where $\alpha_0 \in (0,1)$ and $\alpha_0$ is very close to $1$ i.e., $\widetilde{R}$ is very close to $R^0.$
Then $\widetilde{R}(t) = f(t)\widetilde{R} = f(t)(\alpha_0 R^0_{\textbf{I}}+ R^0_{\textbf{W}})$ is a solution to the ODE \eqref{reaone} for $f$ as above and the product function $\alpha f$ satisfies
\begin{align}\label{function2}
\begin{cases}
(\alpha f)'=\lambda \alpha^2 f^2;\\
\alpha f(0)=\alpha_0.
\end{cases}
\end{align}
Consequently, we obtain $\alpha= \frac{\alpha_0(1- \lambda t)}{1-\lambda \alpha_0 t}$ i.e., $\alpha\rightarrow 0$ as $t\rightarrow \frac{1}{\lambda} < \frac{1}{\lambda\alpha_0}.$

Normalizing the solution curve $\widetilde{R}(t)$ by the norm of its Weyl curvature, one observes that $$\frac{\widetilde{R}(t)}{\widetilde{R}(t)_{\textbf{W}}} =\frac{f(t)}{|f(t)|} \big(\alpha(t)\frac{R^0_{\textbf{I}}}{\|R^0_{\textbf{I}}\|} + \frac{R^0_{\textbf{W}}}{\|R^0_{\textbf{W}}\|}\big) \rightarrow \frac{R^0_{\textbf{W}}}{\|R^0_{\textbf{W}}\|}$$
since $\alpha(t) \rightarrow 0$ and $\frac{f(t)}{|f(t)|}\rightarrow 1$as $t \rightarrow \frac{1}{\lambda}.$ This proves the proposition.
\end{proof}

\begin{rem}\label{thomas2}
Similarly considering the solution curve $\widehat{R}(t)= f(t)( R^0_{\textbf{I}}+ \alpha_0 R^0_{\textbf{W}})$ to the ODE \eqref{reaone} with $R^0, \ \alpha_0$ as above, using similar analysis and normalizing the solution curve by the norm of its identity component, one observes that the normalized curve converges to $\frac{R^0_{\textbf{I}}}{\|R^0_{\textbf{I}}\|}$ which is the identity component of the (normalized) initial point of the solution curve i.e., the identity operator normalized by norm.
\end{rem}
Thus for an Einstein curvature operator $R$ with $s(R) >0$ satisfying $Q(R) =\lambda R,$ one has $\lambda = \frac{s(R)}{n}.$ Then it follows that an Einstein curvature operator $R$ with unit scalar curvature satisfies $Q(R) =\lambda R$ if and only if $R$ is a zero of $\widetilde{Q}.$

In particular, Proposition \ref{thomas} directly proves the unstable behavior of the ODE \eqref{reaone} near the curvature operators of Einstein symmetric spaces which are not space forms and have curvature operators $\mathfrak{R}$ satisfying {small $Q(\mathfrak{R})=\lambda \mathfrak{R}$}, $\lambda >0$.

\subsection{Stability of non-Einstein zeros of $\widetilde{Q}$}

The earlier method fails in proving the stability of a general zero of $\widetilde{Q}.$ In general, using standard ODE technique one notes that in order to prove that an algebraic curvature operator $R$ which is a zero of $\widetilde{Q},$ is unstable, it suffices to show that the derivative operator $D\widetilde{Q}(R): \mathbf{S}_0 \rightarrow \mathbf{S}_0$ has eigenvalues with positive real part.
First we describe the derivative of $\widetilde{Q}$ as follows.

The $O(n)$ invariant irreducible decomposition of {\small $S^2_B(\wedge^2\mathbb{R}^n)$}, gives
$${{\mathbf{S}}}_{0} = \{R \in S^2_B(\wedge^2\mathbb{R}^n) : s(R) = 0\} = \langle {\textbf{Ric}_0} \rangle \oplus \langle {\textbf{W}}\rangle.$$

Let {\small $\;\{\rho_1,\;\rho_2,.....,\rho_{N_1},\;\xi_1,\;\xi_2,....,\xi_{N_2}\}\;$} be an orthonormal basis of  {\small${{\textbf{S}}}_{0}$} such that {\small$\{\rho_{\alpha}\}^{N_1}_{\alpha = 1}$}  and  {\small$ \{\xi_{\beta}\}^{N_2}_{\beta = 1}$}  span  {\small$\langle {\textbf{Ric}_0} \rangle$}  and  {\small$\langle {\textbf{W}}\rangle$} respectively.
Here {\small$N_1,N_2$} denote the respective dimensions of {\small$\langle {\textbf{Ric}_0} \rangle$} and {\small$\langle {\textbf{W}}\rangle.$}

Given $R \in {{\textbf{S}}}_{1},$ it is straightforward to check that
{\small\begin{align}\label{dqrs}
\nonumber &\langle D\widetilde{Q}(R)(\rho_{\alpha_i}),{\rho_{\alpha_j}}\rangle = \textit{tri}(R,{\rho_{\alpha_i}},{\rho_{\alpha_j}}) -\|Ric(R)\|^2{\delta_{ij}} - (n-2){\langle R,{\rho_{\alpha_i}}\rangle}{\langle R,{\rho_{\alpha_j}}\rangle},\;\;{\forall{i,j}}\leq{N_1};\\
\nonumber &\langle D\widetilde{Q}(R)(\rho_{\alpha}),{\xi_{\beta}}\rangle = \textit{tri}(R,{\rho_{\alpha}},{\xi_{\beta}}) - (n-2){\langle R,{\rho_{\alpha}}\rangle}{\langle R,{\xi_{\beta}}\rangle} ;\\
\nonumber &\langle D\widetilde{Q}(R)(\xi_{\beta_k}),{\xi_{\beta_l}}\rangle = \textit{tri}(R,{\xi_{\beta_k}},{\xi_{\beta_l}}) - \|Ric(R)\|^2{\delta_{kl}} \ \ {\forall \ {k,l}} \leq{N_2} \ \ {\rm and}\\
 &\langle D\widetilde{Q}(R)({\xi_{\beta}}),\rho_{\alpha}\rangle = \textit{tri}(R,{\xi_{\beta}},{\rho_{\alpha}}).
\end{align}}
where $\textit{tri}$ is the trilinear form on {\small $S^2_B(\wedge^2\mathbb{R}^n)$} defined in Section 2.
In general, for any {\small $A \in S^2_B(\wedge^2\mathbb{R}^n),$}
{\small \begin{equation}\label{eqgen}
\langle D\widetilde{Q}(R)(A),A \rangle = tri(R,A,A)- \|Ric(R)\|^2 \|A\|^2 - 2\langle Ric(R),Ric(A)\rangle \langle R,A\rangle
\end{equation}}
From the above equations it follows that for an arbitrary {\small $R \in S^2_B(\wedge^2\mathbb{R}^n)$},
{\small $$ \langle D\widetilde{Q}(R)(\rho_{\alpha}),{\xi_{\beta}}\rangle \neq  \langle D\widetilde{Q}(R)({\xi_{\beta}}),\rho_{\alpha}\rangle ,$$}
where {\small $\alpha \in \{1,...,N_1\}$} and {\small $\beta \in \{1,...,N_2\}$}. Thus
{\small$D\widetilde{Q}(R): {{\textbf{S}}}_{0} \rightarrow {{\textbf{S}}}_{0}$} is not self-adjoint in general. This in particular proves that
\begin{lem}
$\widetilde{Q}$ is not a gradient field on {\small $\mathbf{S}_1.$}
\end{lem}

Using \eqref{dqrs} one sees that in particular, when $R$ is of Ricci type: Then, {\small $\langle R,{\xi_{\beta}}\rangle = 0,$ $\forall{\beta}$} and {\small $D\widetilde{Q}(R)$} is self-adjoint. With this observation, we prove that
\begin{prop}\label{ric1}
The curvature operator of $S^n \times \mathbb{R}$ with respect to the standard product metric normalized to unit scalar curvature, is an unstable zero of $\widetilde{Q}$.
\end{prop}
\begin{proof}
The standard product metric $g_p$ on $S^n \times \mathbb{R}$ has constant scalar curvature $n(n-1)$.
Then the normalized metric $g = n(n-1)g_p$ on $S^n \times \mathbb{R}$ is a conformally flat symmetric metric whose curvature operator $\mathfrak{R}$ is a zero of $\widetilde{Q}$ at each point $p \in S^n \times \mathbb{R}$.

Consider any arbitrary point $p =(x,t)\in S^n \times \mathbb{R}.$
Let {\small $ \{e_i\}^{n+1}_{i = 1}$} be an orthonormal basis of the tangent space {\small $T_p(S^n\times \mathbb{R})$} with respect to $g$ such that {\small $\{e_i\}^n_{i=1}$} spans $T_xS^n$.Then one observes that the curvature tensor of $g$ is as follows
\begin{eqnarray*}
(R_g)_{ijkl}=\begin{cases}\delta_{ik} \delta_{jl}\frac{1}{n(n-1)}\,\,\mbox{if}\,\,i\neq j\,\,\mbox{and}\,\, i,j\leq n;\\
0\,\, \ \ \ \ \ \ \ \ \mbox{otherwise}.
\end{cases}
\end{eqnarray*}
Also one has the following
\begin{eqnarray*}
(Ric_g)_{ii} = \begin{cases}\frac{1}{n}\,\, \mbox{if}\,\,\,  i \leq n;\\
0 \,\,\mbox{otherwise}.
\end{cases}
\end{eqnarray*}
\begin{eqnarray*}
(Ric_0)_{ii} = \begin{cases}\frac{1}{n(n+1)}\,\,  \mbox{if}\,\, i \leq n;\\
-\frac{1}{n+1}\,\, \mbox{otherwise}.
\end{cases}
\end{eqnarray*}
i.e., {\small $\|Ric_g\|^2 = tr(Ric_g)^2 = \frac{1}{n}$}.

The curvature operator $\mathfrak{R}$ is of Ricci type given by {\small $\mathfrak{R} = \frac{1}{n(n+1)}{\textbf{I}} + \mathfrak{R}_{\textbf{Ric}_0}$} with {\small $\mathfrak{R}_{\textbf{Ric}_0} = \frac{2}{n-1}Ric_0 \wedge id$}. Here $\mathfrak{R}$ and
{\small $\mathfrak{R}_{\textbf{Ric}_0}$} are both diagonalized by the basis $\{e_i \wedge e_j\}$.
Since {\small $D\widetilde{Q}(\mathfrak{R})$} is self-adjoint, to say that {\small $D\widetilde{Q}(\mathfrak{R})$} has eigenvalues of opposite signs, it suffices to find elements $\xi$ and $\rho$
with {\small $\langle D\widetilde{Q}(\mathfrak{R})(\rho),\rho \rangle >0$} and {\small $\langle D\widetilde{Q}(\mathfrak{R})(\xi),\xi \rangle<0$}.

Now for any {\small $\xi \in \langle \textbf{W} \rangle$} using Proposition \ref{propn} we see
{\small \begin{equation*}
\begin{split} \langle D\widetilde{Q}(\mathfrak{R})(\xi),{\xi}\rangle&
=\textit{tri}(\mathfrak{R},{\xi},{\xi}) -\|Ric_g\|^2\\
& =2\langle Q(\mathfrak{R},{\xi}),{\xi}\rangle -\frac{1}{n} =-\frac{1}{n} <0
\end{split}
\end{equation*}}
This shows that all the directions in {\small $\langle \textbf{W} \rangle$} serve as stable directions for {\small $D\widetilde{Q}(\mathfrak{R}).$}
Similarly, for any {\small $\rho \in \textbf{Ric}_0$} we have
{\small \begin{equation*}
\begin{split}
\langle D\widetilde{Q}(\mathfrak{R})(\rho),{\rho}\rangle&  =\textit{tri}(\mathfrak{R},{\rho},{\rho}) -\|Ric_g\|^2 - (n-1){\langle
\mathfrak{R},{\rho}\rangle}^2\\
& =\textit{tri}(\mathfrak{R},{\rho},{\rho}) -\frac{1}{n} - (n-1){\langle \mathfrak{R},{\rho}\rangle}^2
\end{split}
\end{equation*}}
Let {\small $\rho = id \wedge A$}, where $A$ is the traceless diagonal matrix given by {\small $A = 2\sqrt{\frac{n}{n^2-1}}(D_n + \frac{n-1}{n}D_{n-1} +....+\frac{2}{n}D_2 + \frac{1}{n}D_1)$}, where for each $k$, $(1\leq k \leq n)$; $D_k$  is the traceless diagonal matrix of order $n+1$ with $1,-1$ respectively at
$k$th and $(k+1)$th diagonal places and zero elsewhere. Clearly, $\rho$ is a unit vector in {\small $\langle \textbf{Ric}_0\rangle$} and
{\small \begin{equation*}
\begin{split}
\langle D\widetilde{Q}(\mathfrak{R})(\rho),{\rho}\rangle&
= 2\langle Q(\mathfrak{R},{\rho}),{\rho}\rangle -\frac{1}{n} - (n-1){\langle \mathfrak{R}_{\textbf{Ric}_0},{\rho}\rangle}^2\\
& =\frac{n-1}{n(n+1)} + \frac{2(n-1)}{n(n+1)} - \frac{1}{n} - \frac{2}{n(n^2-1)}\\
& =\frac{2n(n-3)+2}{n^2-1} >0\,\,\forall n\geq3.
\end{split}
\end{equation*}}
Thus for $n\geq 3$, $D\widetilde{Q}(\mathfrak{R})$ has eigenvalues of both signs. This shows that there exist solution curves to the associated ODE \eqref{reaone} emerging from $\mathfrak{R}$ above and moving away and the proposition
follows.
\end{proof}
Again this shows the unstable behavior of the reaction ODE \eqref{reaone} near the product curvature operator of {\small $S^n \times \mathbb{R}.$}

Let $R$ be the curvature operator of the product space $M \times \mathbb{R}^k$ where $M$ is an Einstein symmetric space with positive scalar curvature and not a space form. We prove the following result
as a Corollary to Proposition \ref{thomas}
\begin{cor}\label{cor}
Let {\small $\breve{R} \in S^2_B(\wedge^2\mathbb{R}^{n+k})$} be a product curvature operator of the form
{$$\breve{R}= \left(
               \begin{array}{cc}
                 \widehat{R} & 0 \\
                 0 & 0 \\
               \end{array}
             \right)
$$}
where {\small $\widehat{R}\in S^2_B(\wedge^2\mathbb{R}^{n}$)} is an Einstein algebraic curvature operator with positive scalar curvature satisfying {\small $Q(\widehat{R}) = \lambda \widehat{R}$} for some $\lambda >0.$ Assume further that $\widehat{R}$ is not a multiple of identity operator. Then the ODE \eqref{reaone}
behaves unstably near $\breve{R}$. In fact,  with suitable normalization there exists a solution curve to  \eqref{reaone} emerging from a point near $\breve{R}$ and approaching a curvature operator $R_{\lambda}$ which is of the form
{\small $$R_{\lambda}= \left(
               \begin{array}{cc}
                 \frac{\widehat{R}_{\textbf{W}}}{\|\widehat{R}_{\textbf{W}}\|} & 0 \\
                 0 & 0 \\
               \end{array}
             \right).$$}
\end{cor}
\begin{proof}
Using Corollary \ref{coro} it follows that for a curvature operator $R$ of the form {\small $$R = \left(
               \begin{array}{cc}
                R_1 & 0 \\
                 0 & 0 \\
               \end{array}
             \right),$$}
$Q(R)$ is of the form {\small $$Q(R) = \left(
               \begin{array}{cc}
               Q(R_1) & 0 \\
                 0 & 0 \\
               \end{array}
             \right).$$}
The rest of the proof goes along the similar lines as Proposition \ref{thomas}.
\end{proof}
Using Proposition \ref{thomas} and Corollary \ref{cor} one concludes that
\begin{thm}
The reaction ODE \eqref{reaone} behaves unstably near the curvature operator of $M \times \mathbb{R}^k$ for $k \geq 0,$ where $M$ is an Einstein symmetric space of positive scalar curvature and $M$ is not a spherical pace form.
\end{thm}
Note that the Corollary \ref{cor} does not hold when $R$ is the curvature operator of $S^n \times \mathbb{R}^k,$ $k \geq 1$.
A proof of the unstable behavior of $Q$ near curvature operator of $S^n \times \mathbb{R}^k,$ $k \geq 1;$ is given by the following proposition
\begin{prop}\label{spaceform}
The behavior of the reaction ODE \eqref{reaone} is unstable near the product curvature operator of $S^n \times \mathbb{R}^k$ with respect to the standard product metric.
\end{prop}

It will be evident from the proof of the above proposition that in particular,
there exists solution curves to the ODE \eqref{reaone} which emerge from the curvature operators of $S^n \times \mathbb{R}^k$ and approach the curvature operator of the spherical space form of dimension $n+k$.

\begin{proof}
We follow Hamilton's method \cite{H3} to conclude the desired result.
Consider an orthonormal basis $\mathfrak{B} = \{e_i \wedge e_j\}$ of {\small $\wedge^2\mathbb{R}^{n+k}$} where $\{e_i\}_{i=1}^{n+k}$ is an orthonormal basis of {\small $\mathbb{R}^{n+k} = \mathbb{R}^n \oplus \mathbb{R}^k$} such that $\{e_i\}_{i=1}^{n}$ and $\{e_i\}_{i=n+1}^{n+k}$ span $\mathbb{R}^n$ and $\mathbb{R}^k$ respectively. Let {\small $R \in S^2_B(\wedge^2\mathbb{R}^{n+k})$} be diagonalized by the above basis as follows
\begin{equation}\label{sphere}
R(e_i \wedge e_j) = \begin{cases} x e_i \wedge e_j \ \ {\rm if} \ \ 1\leq i < j \leq n;\\
                                  y e_i \wedge e_j \ \ {\rm if} \ \ i \leq n, \ j \geq n+1;\\
                                  z e_i \wedge e_j \ \ {\rm if} \ \  i,  j \geq n+1.
\end{cases}
\end{equation}
Note that the curvature operator of $ S^n \times \mathbb{R}^k$ with respect to the product metric satisfies the above equations for $x = 1, \ y=z= 0.$
Also the curvature operator of $S^{n+k}$ with respect to the round metric satisfies the above equations for $x =y=z= 1.$

Identifying {\small $\wedge^2\mathbb{R}^{n+k}$} with the Lie algebra $\mathfrak{so}(n+k)$ canonically and using the standard Lie bracket relations among the basis elements of $\mathfrak{B}$ it follows that if $R$ satisfies \eqref{sphere}, then $Q(R)$ is of the following form
\begin{align}
Q(R)(e_i \wedge e_j) = \begin{cases} \big((n-1)x^2 + ky^2 \big) e_i \wedge e_j \ \ \ \ \ \ \ \ \ \ \ \ \ \ {\rm if} \ \ 1\leq i < j \leq n;\\
                                     y \big(y + (n-1)x+ (k-1)z\big) e_i \wedge e_j \ \ {\rm if} \ \ i \leq n, \ j \geq n+1;\\
                                    (k-1)z^2 + ny^2 e_i \wedge e_j \ \  \ \ \ \ \ \ \ \ \ \ \ \ \ \ \ {\rm if} \ \ i,  j \geq n+1.
\end{cases}
\end{align}
Thus ODE \eqref{reaone} reduces to the following system of equations
\begin{align}
\begin{cases}
\frac{dx}{dt} = (n-1)x^2 + ky^2 \\
\frac{dy}{dt} = y(y + (n-1)x + (k-1)z)\\
\frac{dz}{dt} = (k-1)z^2 + ny^2
\end{cases}
\end{align}
The above system is of the form
{\small $\frac{dV}{dt}= \phi(V)$}
where $\phi(V)=((n-1)x^2 + ky^2, y(y + (n-1)x +(k-1)z)$ and $V=(x,y,z)$.
Define $\rho(V)= y + (n-1)x + (k-1)z$ and consider the associated system
{\small $\frac{dV}{dt}= \phi(V)-\rho(V)V.$}

In the associated system $\frac{dy}{dt}=0.$ In particular putting $y =1$ the associated system becomes
\begin{align}\label{normalized}
\begin{cases}
\frac{dx}{dt} = k - x - (k-1)xz\\
\frac{dz}{dt} = n - z - (n-1)xz\\
y =1
\end{cases}
\end{align}
which is same as
\begin{align*}
\begin{cases}
\frac{d\breve{x}}{dt} = \frac{2nk-(n+k)}{n-1} - \breve{x} - \frac{1}{2}(n-1)(\breve{x}^2 - \breve{z}^2)\\
\frac{d\breve{z}}{dt} = \frac{n-k}{n-1} -\breve{z} \\
y =1
\end{cases}
\end{align*}
where $\breve{x}= x+ \frac{k-1}{n-1}z, \ \ {\rm and} \ \ \breve{z}=x- \frac{k-1}{n-1} z.$

The only solution to this system corresponds to  $\breve{z} \rightarrow  \frac{n-k}{n-1}, \ \ \breve{x} \rightarrow \frac{n+k-2}{n-1}$ i.e., $$x+ \frac{k-1}{n-1}z \rightarrow \frac{n+k-2}{n-1}, \ \ x- \frac{k-1}{n-1}z  \rightarrow \frac{n-k}{n-1}$$ which corresponds to the following solution $$x \rightarrow 1, \ \ z \rightarrow 1, \ \ y=1$$ of the associated system \eqref{normalized}.
This clearly says that there exist solution curves to the original system with $\frac{x}{y}, \frac{z}{y} \rightarrow 1$ i.e., there exists solution curve to the reaction ODE \eqref{reaone}
which emerges from the curvature operator of $S^n \times \mathbb{R}^k$ and approaches the curvature operator of the round sphere $S^{n+k}$.

\end{proof}
\begin{rm}
Proposition \ref{spaceform} and Corollary \ref{cor} together complete the proof of Theorem \ref{main}.
\end{rm}


\begin{thebibliography}{10}
\baselineskip=17pt

\bibitem{Bes} Arthur L. Besse, \emph{Einstein manifolds}, Ergebnisse der Mathematik und ihrer Grenzgebiete (3)[Results in Mathematics and Related Areas (3)], Volume {\bf 10}, Springer-Verlag, Berlin (1987).

\bibitem{BW} C. B\"ohm and B. Wilking, \emph{Manifolds with positive curvature operators are space forms}, Ann. of Math.(2) \textbf{167} (2008), no. 3, 1079-1097.

\bibitem{BS} S. Brendle and R. Schoen,  \emph{Manifolds with 1/4-pinched curvature are space forms}, J. Amer. Math. Soc. \textbf{22} (2009), no. 1, 287-307.

\bibitem{H1} R. Hamilton, \emph{Three-manifolds with positive Ricci curvature}, J. Differential Geom. \textbf{17} (1982), no. 2, 255–306.

\bibitem{H2} R. Hamilton, \emph{Four-manifolds with positive curvature operator}, J. Differential Geom. \textbf{24} (1986), no. 2, 153-179.

\bibitem{H3} R. Hamilton, \emph{The formation of singularities in the Ricci flow},  Surveys in differential geometry, Volume \textbf{II} (Cambridge, MA, 1993), Int. Press, Cambridge, MA, 1995, 7-136.

\bibitem{Hu} G. Huisken, \emph{Ricci deformation on the metric on a Riemannian manifold}, J. Differential Geom. \textbf{21} (1985), 47-62.

\bibitem{PT} P. Topping, \emph{Lectures on the Ricci flow}, London Mathematical Society Lecture Note series, No. 325 (2006), Cambridge University Press.
\end{thebibliography}
\end{document}